\documentclass[12 pt]{amsart}
\usepackage{amssymb}
\usepackage{amsmath}
\usepackage{amscd}
\usepackage{graphicx}
\usepackage{latexsym}

\calclayout\evensidemargin .6in

\def\R{\mathbb{R}}
\def\Z{\mathbb{Z}}

\newtheorem{theorem}{Theorem}[section]
\newtheorem{lemma}[theorem]{Lemma}

\newtheorem{corollary}[theorem]{Corollary}

\theoremstyle{definition}

\newtheorem{remark}[theorem]{Remark}

\title[CONTACT STRUCTURES ON $G_2$-MANIFOLDS AND SPIN 7-MANIFOLDS]
{CONTACT STRUCTURES ON $G_2-$MANIFOLDS \\ AND SPIN 7-MANIFOLDS}

\author{M. Firat Arikan}
\address{Department of Mathematics, University of Rochester, Rochester NY, USA
\indent Max Planck Institute for Mathematics, Bonn, GERMANY}
\email{arikan@math.rochester.edu, arikan@mpim-bonn.mpg.de}
\thanks{The first named author is partially supported by NSF FRG grant DMS-1065910}

\author{Hyunjoo Cho}
\address{Department of Mathematics, University of Rochester, Rochester NY, USA}
\email{cho@math.rochester.edu}

\author{Sema Salur}
\address{Department of Mathematics, University of Rochester, Rochester NY, USA}
\email{salur@math.rochester.edu}
\thanks{The third named author is partially supported by NSF grant DMS-1105663}


\begin{document}
\begin{abstract}
We show that there exist infinitely many pairwise distinct non-closed $G_2$-manifolds (some of which have holonomy full $G_2$) such that they admit co-oriented contact structures and have co-oriented contact submanifolds which are also associative. Along the way, we prove that there exists a tubular neighborhood $N$ of every orientable three-submanifold $Y$ of an orientable seven-manifold with spin structure such that for every co-oriented contact structure on $Y$, $N$ admits a co-oriented contact structure such that $Y$ is a contact submanifold of $N$. Moreover, we construct infinitely many pairwise distinct non-closed seven-manifolds with spin structures which admit co-oriented contact structures and retract onto co-oriented contact submanifolds of co-dimension four.
\end{abstract}

\maketitle

%
%

\section{Introduction}

Given a smooth manifold $M$ of dimension seven, one can ask the existence of two different structures, namely, a $G_2$ structure and a contact structure: We say that $M$ admits a {\it $G_2$-structure} if the structure group of $TM$ can be reduced to $G_2$, which is the exceptional Lie group of all linear automorphisms of the imaginary octonions $im \mathbb{O}\cong \mathbb{R}^7$ preserving a certain cross product. On the other hand, a \emph{(co-oriented) contact structure} on $M$ is a co-oriented $6$-plane distribution which is totally non-integrable. The induced geometries on $M$ have quite different properties. For instance, $G_2$-structure determines a unique metric on $M$, and so there are local invariants in $G_2$ geometry. Whereas in contact geometry there are no local invariants because any point has a standard neighborhood by Darboux's theorem. It is known that there are $7$-manifolds admitting both $G_2$ and a contact structures, and that these structures are compatible in a certain way \cite{Arikan-Cho-Salur}. Here we present new examples of such $7$-manifolds with the absence of compatibility.

\vspace{.1in}

We can describe $G_2$ geometry more precisely in the following way: Identify the group $G_2$ as the subgroup of $GL(7,\mathbb{R})$ which preserves the $3$-form $$\varphi_0=e^{123}+e^{145}+e^{167}+e^{246}-e^{257}-e^{347}-e^{356}$$
where $(x_1,...,x_7)$ are the coordinates on $\mathbb{R}^7$ and $e^{ijk}=dx^i \wedge dx^j \wedge dx^k$. As an equivalent definition, a manifold with a $G_2$-structure $\varphi$ is a pair $(M,\varphi)$,
where $\varphi$ is a $3$-form on $M$, such that $(T_pM,\varphi)$ is isomorphic to $(\mathbb{R}^7,\varphi_0)$ at every point $p$ in $M$. Such a $\varphi$ defines a Riemannian metric $g$ on $M$. We say $\varphi$ is \emph{torsion-free} if $\nabla \varphi=0$ where
$\nabla$ is the Levi-Civita connection of $g$. The latter holds if and only if $d\varphi=d(\ast\varphi)=0$ where $``\ast"$ is the Hodge star operator defined by the metric $g$. A Riemannian manifold with a torsion free $G_2$-structure is called a {\it $G_2$-manifold}. Equivalently, the pair $(M,\varphi)$ is called a $G_2$-manifold if its holonomy group (with respect to $g$) is a subgroup of $G_2$. Finally, a three dimensional submanifold $Y^3$ of a manifold $M$ with (torsion-free) $G_2$-structure $\varphi$ is said to be {\it associative} if $\varphi$ is a volume form on $Y$. More details can be found in \cite{Bryant}, \cite{Bryant2}, \cite{Harvey-Lawson}  and \cite{Joyce}.

\vspace{.1in}

Contact structures are defined in any odd dimension $2n+1$ for $n\geq1$. Here we only consider co-oriented contact structures: A \emph{contact form} on a smooth $(2n+1)$-dimensional manifold $M$
is a $1$-form $\alpha$ such that $\alpha \wedge (d\alpha)^n \neq0$ (i.e., $\alpha \wedge (d\alpha)^n$ is a volume form on $M$). The \emph{Reeb vector field} of a contact form $\alpha$ is defined to be the unique global nowhere-zero vector field $R$ on $M$ satisfying the equations
\begin{equation} \label{eqn:Defining_Reeb}
\iota_Rd\alpha=0, \quad \alpha(R)=1
\end{equation}
where $`` \iota "$ denotes the interior product. The hyperplane field (of rank $2n$) $\xi=\textrm{Ker} (\alpha)$ of a contact form $\alpha$ is called a \emph{(co-oriented) contact structure} on $M$. The pair $(M,\xi)$ (or sometimes $(M,\alpha)$) is called a \emph{contact manifold}. A submanifold $Y \subset (M,\xi)$  is said to be a \emph{contact submanifold} if $TY \cap \xi|_{Y}$ defines a contact structure on $Y$. We say that two contact manifolds $(M_1,\xi_1)$ and $(M_2,\xi_2)$ are \emph{contactomorphic} if there exists a diffeomorphism $f:M_1\longrightarrow M_2$ such that $f_\ast(\xi_1)=\xi_2$. Also a \emph{strict contactomorphism} between two contact manifolds $(M_1,\alpha_1)$ and $(M_2,\alpha_2)$ is a contactomorphism $f:M_1\longrightarrow M_2$ such that $f^*(\alpha_2)=\alpha_1$.

\vspace{.1in}

The equation (\ref{eqn:Defining_Reeb}) implies that the Reeb vector field $R$ co-orients $\xi$ and, as a result, the structure group of the tangent frame bundle can be reduced to $U(n) \times 1$. Such a reduction of the structure group is called an \emph{almost contact structure} on $M$. By a result of \cite{Arikan-Cho-Salur}, every $7$-manifold with spin structure (and hence every manifold with $G_2$-structure) admits an almost contact structure.

\vspace{.1in}

As an alternative but equivalent definition, for $R,\alpha$ and $\xi$ as above, the triple $(J,R,\alpha)$ is called an \emph{associated almost contact structure} for $\xi$ if $J$ is $d\alpha$-compatible almost complex structure on $\xi$. Furthermore, if $g$ is a metric on $M$ satisfying
$$g(JX,JY)=g(X,Y)-\alpha(X)\alpha(Y) \quad \textrm{and} \quad d\alpha(X,Y)=g(JX,Y)$$
for all $X,Y \in TM$, then it is called an \emph{associated metric} or sometimes \emph{contact metric}. We refer the reader to \cite{Blair}, \cite{Geiges} for more on contact geometry.

\vspace{.1in}

Now suppose that $(M,\varphi)$ is a manifold with $G_2$-structure with an associative submanifold $Y$. One interesting question is: Does $M$ admit a contact structure $\eta$ such that $Y$ is a contact submanifold of $(M,\eta)$ ? Another reasonable one is: Given a contact structure $\xi$ on $Y$ can we extend it to some contact structure $\eta$ on $M$ so that $(Y,\xi)$ is a contact submanifold of $(M,\eta)$ ?

\vspace{.1in}

In this paper, we will show that both questions above have positive answers by constructing such manifolds. As we will see in Section \ref{sec:Main results}, one can extend any contact structure on $Y$ to some contact structure on its tubular neighborhood in $M$, and using this fact and a result of \cite{Robles_Salur} we obtain co-oriented contact structures on non-closed $G_2$-manifolds having co-oriented contact submanifolds which are also associative. Before studying the $G_2$ case we will first have similar extension and existence results for spin $7$-manifolds and manifolds with $G_2$-structures.

\medskip \noindent {\em Acknowledgments.\/} The authors would like to
thank Selman Akbulut and Jonathan Pakianathan for helpful conversations and remarks.

%
%

\section{Preparatory results} \label{sec:Preparatory results}

To be able show the existence of contact structures on certain family of spin $7$-manifolds, we first need some contact geometry preparation.

\begin{lemma} \label{lem:Contact_Str_on_Y cross_R4}
Let $Y$ be any smooth $3$-dimensional manifold. Then for any co-oriented contact structure $\xi$ on $Y$, there exists a co-oriented contact structure $\eta$ on $Y \times \R^{4}$ such that $Y \times \{0\}$ is a contact submanifold of $(Y \times \R^{4},\eta)$.
\end{lemma}

\begin{proof}
Since $\xi$ is co-oriented, there is a contact form $\alpha$ on $Y$ such that $\xi=\textrm{Ker}(\alpha)$. Also consider the $1$-from $\lambda=x_1 \wedge dy_1 + x_2 \wedge dy_2$ on $\R^4$ with the standard coordinates $(x_1,y_1,x_2,y_2)$. Let $p_1:Y \times \R^4 \to Y$ and $p_2:Y \times \R^4 \to \R^4$ denote the usual projections.
Now consider the pullback forms on $Y \times \R^{4}$ given by $\tilde{\alpha}:=p_1^{\ast}(\alpha)$, $\tilde{\lambda}:=p_2^{\ast}(\lambda)$ and set
$$\beta=\tilde{\alpha}+\tilde{\lambda}.$$
Then
$d\beta=d\tilde{\alpha}+d\tilde{\lambda}$. Also using the facts $(d\tilde{\alpha})^i=0, \; \forall i\geq 2$, $(d\tilde{\lambda})^j=0, \; \forall j\geq 3$ we compute $$\beta \wedge (d\beta)^{3}=
3 \;\tilde{\alpha} \wedge d\tilde{\alpha} \wedge (d\tilde{\lambda})^2.$$ Therefore, $\beta \wedge (d\beta)^{3}$ is a volume form on $Y \times \R^4$ because $\alpha \wedge d\alpha$ and $(d\lambda)^2=dx_1\wedge dy_1\wedge dx_2\wedge dy_2 $ are volume forms on $Y$ and $\R^4$, respectively. Hence, $\beta$ is a contact form which defines a co-oriented contact structure $\eta$ on $Y \times \R^4$. \\

For the last statement, consider the embedding $f:Y \longrightarrow f(Y) \subset Y \times \R^4$ given by $y \mapsto (y,0)$ of $Y \times \{0\}$ in $Y \times \R^4$. Then we compute $$f^*(\beta)=f^*(\tilde{\alpha})+f^*(\tilde{\lambda})
=f^*(p_1^*(\alpha))+\underbrace{\tilde{\lambda}|\,_{Tf(Y)}}_{=0}=(\underbrace{p_1 \circ f}_{id})^*(\alpha)=\alpha$$
which shows that $\phi: (Y,\alpha) \to (f(Y),\beta|\,_{Tf(Y)})$ is a strict contactomorphism. In particular, $(Y \times \{0\},\phi_*(\xi))$ (or equivalently, $(Y,\xi)$) is a contact submanifold of $(Y \times \R^4,\eta)$.

\end{proof}

\begin{remark}
The above lemma can be proved in a much more general setting. More precisely, if we take $Y$ to be an arbitrary $(2n+1)$-dimensional manifold ($n\geq3$) admitting a co-oriented contact structure and $\R^4$ to be $\R^{2m}$ ($m\geq2$), then a similar statement is still true, i.e., there exist a co-oriented contact structure on $Y \times \R^{2m}$ such that $Y \times \{0\}$ is a contact submanifold. Since we are interested in dimension seven, the version stated above will be enough for our purposes.
\end{remark}

Next, we present a basic homotopy theoretical fact as a lemma below. Recall that a manifold $M$ admits a spin structure if and only if the second Stiefel-Whitney class $w_2(M)$ of $M$ vanishes.

\begin{lemma} \label{lem:Normal_bundle_trivial}
Let $Y$ be any orientable smooth $3$-manifold embedded in an orientable $7$-manifold $M$ with a spin structure. Then there exists a tubular neighborhood of $Y$ in $M$ which is trivial.
\end{lemma}

\begin{proof}
Since $3$-dimensional manifold $Y$ is orientable, it is parallelizable \cite{Steenrod, Stiefel}, in other words, the tangent bundle $TY$ is trivial. Therefore, $w_i(Y)=0, \,\forall i>0$. Let $N$ be the normal bundle of $Y$ in $M$. Then by the tubular neighborhood theorem, it suffices to show that $N$ is trivial. Note that $TM|_Y=TY \oplus N$ and so the total Stiefel-Whitney classes satisfies the equation $$w(TM|_Y)=w(TY)w(N).$$ which gives $w(N)=w(TM|_Y)$ as $w(TY)=1$. Thus, $w_1(N)=w_2(N)=0$ because $M$ is orientable and spin.

The vector bundle $N$ (of rank $4$) over $Y$ is classified by a Gauss map $$g: Y \to BO(4)$$ which lifts to a map $f: Y \to BSpin(4)$ as $w_1(N)=w_2(N)=0$. Note that the domain of $f$ is the $3$-dimensional manifold $Y$, and so we are allowed to change the codomain $BSpin(4)$ of $f$ by adding $k$-cells for $k\geq 5$, and hence we can assume that the homotopy groups $\pi_k$ of the codomain vanish for $k\geq4$. Call this new codomain $X$, and so we can rewrite $f$ as $$f:Y \to X \quad \textrm{with} \quad \pi_k(X)=0, \quad \forall k \geq 4.$$
On the other hand, since $Spin(4) = SU(2) \times SU(2) = S^3 \times S^3$, its lowest nontrivial homotopy group is $\pi_3=\Z \times \Z$. Therefore, the lowest nontrivial homotopy group of $BSpin(4)$ is  $\pi_4=\Z \times \Z$ which was already killed in the construction of $X$ above. We conclude that the codomain $X$ is contractible. As a result, $f:Y \to X$ (and so $g:Y \to BO(4)$) is homotopically trivial. Equivalently, $N$ is a trivial vector bundle.

\end{proof}

%
%

\section{Main results} \label{sec:Main results}

In this section we will first focus on spin $7$-manifolds and manifolds with $G_2$-structures, and then we will prove our two main results which can be stated as follows:

\begin{theorem} \label{thm:Existence_contact_on_G_2_mfld}
There exist infinitely many pairwise distinct non-closed $G_2$-manifolds such that they admit co-oriented contact structures and have co-oriented contact submanifolds which are also associative. Furthermore, the restrictions of the $G_2$ metrics on the submanifolds are associated metrics of the contact structures on the submanifolds.
\end{theorem}

\begin{theorem} \label{thm:Existence_contact_on_G_2_mfld_Full_G_2}
There are non-closed $G_2$-manifolds with holonomy exactly $G_2$ such that they admit co-oriented contact structures and have co-oriented contact submanifold which are also associative. Furthermore, the restrictions of the $G_2$ metrics on the submanifolds are associated metrics of the contact structures on the submanifolds.
\end{theorem}

Using the results of the previous section we immediate conclude that

\begin{theorem} \label{thm:Contact_Str_on_Tub_Ngbd_Spin_Case}
Let $Y$ be any orientable smooth $3$-manifold embedded in an orientable $7$-manifold $M$ with a spin structure. Then for any co-oriented contact structure $\xi$ on $Y$, there is a co-oriented contact structure $\eta$ on some tubular neighborhood $N$ of $Y$ in $M$ such that $(Y,\xi)$ is a contact submanifold of $(N,\eta)$.
\end{theorem}

\begin{proof}
We know, by Lemma \ref{lem:Normal_bundle_trivial}, that a small enough tubular neighborhood, say $N$, of $Y$ in $M$ is trivial. Therefore, we may write $N\approx Y \times \R^4$. (note that $Y \subset N$ is identified with $Y \times \{0\} \subset Y\times \R^4$). On the other hand, Lemma \ref{lem:Contact_Str_on_Y cross_R4} implies that for any co-oriented contact structure $\xi$ on $Y$, there exists a co-oriented contact structure $\eta$ on $Y\times \R^4$ (and so on $N$) such that $(Y,\xi)$ is a contact submanifold of $(N\approx Y\times \R^4,\eta)$.

\end{proof}

Since every manifold with $G_2$-structure admits a spin structure, we have the following corollary:

\begin{corollary} \label{cor:Contact_Str_on_Tub_Ngbd_G2_Case}
Let $Y$ be any orientable smooth $3$-manifold embedded in a manifold $M$ with $G_2$-structure. Then for any co-oriented contact structure $\xi$ on $Y$, there exists a co-oriented contact structure $\eta$ on some tubular neighborhood $N$ of $Y$ in $M$ such that $(Y,\xi)$ is a contact submanifold of $(N,\eta)$. \qed
\end{corollary}

Now observe that with a little more care in the proof of Lemma \ref{lem:Normal_bundle_trivial}, we can see that a trivial neighborhood of an orientable $3$-manifold in an orientable spin $7$-manifold is itself a manifold with a spin structure. Indeed, we can also proceed as follows: Let $Y$ be any closed orientable smooth $3$-manifold and consider the product $M=Y \times \R^4$. Then since $TM=TY \times T\R^4$ and from the fact that $w_2(Y)=0$ ($Y$ is parallelizable), we have $w_2(M)=0$. Therefore, $M$ admits a spin structure. On the other hand, $Y$ admits a co-oriented contact structure $\xi$ \cite{Lutz}, \cite{Martinet}, \cite{Thurston-Winkelnkemper}. Also we know by Lemma \ref{lem:Contact_Str_on_Y cross_R4} that for any co-oriented contact structure $\xi$ on $Y$ we can construct a co-oriented contact structure on $Y \times \R^4$ such that $(Y,\xi)$ is a contact submanifold. Also note that distinct (non-homeomorphic) $Y$'s gives distinct products. We can summarize this paragraph as the following theorem.

\begin{theorem} \label{thm:Contact_Str_on_Tub_Ngbd_Spin_Case}
There exist infinitely many pairwise distinct non-closed $7$-manifolds with spin structure such that they admit co-oriented contact structures and retract onto co-oriented contact submanifolds of co-dimension four. \qed
\end{theorem}

\begin{proof}[Proof of Theorem \ref{thm:Existence_contact_on_G_2_mfld}]
We start our proof by recalling a useful theorem proved in \cite{Robles_Salur} which is, in fact, the key result for the proof.

\begin{theorem} [\cite{Robles_Salur}] \label{thm:Robles_Salur}
Assume $(Y^3,g)$ is a closed, oriented, real analytic Riemannian $3$-manifold. Then there exists a $G_2$-manifold $(M^7,\varphi)$ and an isometric embedding $i: Y \rightarrow M$ such that the image $i(M)$ is an associative submanifold of $M$. Moreover, $(M,\varphi)$ can be chosen so that $i(Y)$ is the fixed
point set of a nontrivial $G_2$-involution $r : M \rightarrow M$. Moreover, as long as the metric $g$ on $Y$ is not flat, the holonomy of $M$ is exactly $G_2$.
\end{theorem}

Let $Y$ be any closed, oriented, real analytic $3$-manifold. Then as before we know  \cite{Lutz}, \cite{Martinet},\cite{Thurston-Winkelnkemper} that there exists a co-oriented contact structure $\xi$ on $Y$. Let $g$ be a metric on $Y$ associated to $\xi$. Then applying Theorem \ref{thm:Robles_Salur} to the Riemannian manifold $(Y,g)$, we know that there exists a $G_2$-manifold $(M,\varphi)$ such that $Y$ is an associative submanifold and the $G_2$-metric $g_\varphi$ restricts to $g$ on $Y$. On the other hand, the construction used in the proof of Theorem \ref{thm:Robles_Salur} also implies that $M$ is, indeed, topologically a thickening of $Y$ (in other words, topologically $M\approx Y \times \R^4$). Therefore, by Lemma \ref{lem:Contact_Str_on_Y cross_R4}, there exists a co-oriented contact structure $\eta$ on $M\approx Y \times \R^4$ such that $(Y,\xi)$ is a contact submanifold of $(M,\eta)$.

\vspace{.1in}

Moreover, clearly we have infinitely many distinct choices for the $3$-manifold $Y$, and so we can construct infinitely many pairwise distinct $G_2$-manifolds with the properties described above. This finishes the proof.

\end{proof}

\begin{proof}[Proof of Theorem \ref{thm:Existence_contact_on_G_2_mfld_Full_G_2}]
We basically follow the same steps used in the previous proof, with some exceptions. More precisely, since we want to construct $G_2$-manifolds with holonomy exactly $G_2$, the choice of $Y$ is more restrictive. By Theorem \ref{thm:Robles_Salur}, if the metric $g$ on $Y$ is not flat, then $M$ has holonomy $G_2$. So we need to find a closed, oriented, real analytic $3$-manifold admitting a contact structure whose associated metric is not flat. To this end, we recall the fact that there exist closed orientable $3$-manifolds which admit no flat contact metrics \cite{Rukimbira}. Therefore, if $Y$ is chosen to be one of these $3$-manifolds, then for any contact structure on $Y$, its associated metric (i.e., contact metric) is not flat. Hence, the result follows.
\end{proof}



\clearpage
\begin{thebibliography}{999999}

\bibitem{Arikan-Cho-Salur}
M. F. Arikan, H. Cho, S. Salur, \emph{Existence of Compatible Contact Structures on $G_2$-manifolds}, arXiv:1112.2951v1. (To appeare in Asian Journal of Mathematics)

\bibitem{Blair} D. Blair,
{\em Contact Manifolds in Riemannian Geometry}, Lecture Notes in
Mathematics 509, Springer-Verlag, 1976.

\bibitem{Bryant}
R.~Bryant, \emph{Metrics with exceptional holonomy}, Ann. of
Math. 126 (1987), 525--576.

\bibitem{Bryant2}
R.~Bryant, \emph{On the construction of some complete metrics
with exceptional holonomy}, Duke math. J. 58, no. 3(1989), 829--850

\bibitem{Geiges} H. Geiges, {\em Contact geometry},  Handbook of differential
geometry. Vol. II, 315--382, Elsevier/North-Holland, Amsterdam,
2006.

\bibitem{GrayJW}
J.~W.~Gray, \emph{Some global properties of contact structures},
Ann. of Math. 69 no 2(1959), 421--450.

\bibitem{Harvey-Lawson} F. Harvey and H. Lawson,
{\it Calibrated Geometries}, Acta. Math. {\bf 148} (1982), 47--157.

\bibitem{Joyce}
D.~Joyce, \emph{The exceptional holonomy groups and calibrated
geometry}, Proceedings of $11^{th}$ Gokova Geometry-Topology
Conference, (2003), 1--29.

\bibitem{Lutz} R. Lutz,
{\em Sur quelques propri\'et\'es des formes diff\'erentielles en dimension trois},
Th\'ese, Strasbourg, 1971.

\bibitem{Martinet} J. Martinet,
{\em Formes de contact sur les vari\'et\'es de dimension 3},
Lect. Notes in Math, 209 (1971), 142-163

\bibitem{Robles_Salur} C. Robles, S. Salur,
\emph{Calibrated associative and Cayley embeddings},
Asian J. Math. 13 (2009), no. 3, 287–305.

\bibitem{Rukimbira} P. Rukimbira,
{\em A characterization of flat contact metric geometry},
Houston J. Math. 24 (1998), no. 3, 409--414.

\bibitem{Steenrod}
N. Steenrod, \emph{The topology of fiber bundles}, Princeton
Universtiy Press, 1951.

\bibitem{Stiefel}
E. Stiefel, \emph{Richtungsfelder und Fernparallelismus in Mannigfaltigkeiten}, Comment. Math. Helv. 8 (1936) 3-51.

\bibitem{Thurston-Winkelnkemper} W. Thurston, H. Winkelnkemper,
{\em On the existence of contact forms}, Proc. Amer. Math. Soc.
{\bf 52} (1975), 345--347.

\end{thebibliography}
\end{document}